\documentclass[10pt,a4paper]{article}
\usepackage[utf8x]{inputenc}
\usepackage{ucs}
\usepackage{amsmath}
\usepackage{amsfonts}
\usepackage{amssymb}
\usepackage{tikz}
\usetikzlibrary{arrows}
\usepackage{amsthm}
\usepackage[boxed,vlined,linesnumbered]{algorithm2e}
\theoremstyle{plain} \newtheorem{lem}{Lemma}[section]
\theoremstyle{plain} \newtheorem{prop}{Proposition}[section]
\theoremstyle{plain} \newtheorem{thm}{Theorem}[section]
\theoremstyle{definition} \newtheorem{defi}{Definition}[section]
\author{Wenjie Fang, Roberto Mantaci \\ LIAFA, Universit\'e Paris Diderot - Paris 7}
\title{A Recursive Structure of Sand Pile Model and Its Applications}
\begin{document}
\maketitle

\begin{abstract}
The Sand Pile Model (\mbox{SPM}) and its generalization, the Ice Pile Model (\mbox{IPM}), originate from physics and have various applications in the description of the evolution of granular systems. In this article, we deal with the enumeration and the exhaustive generation of the accessible configuration of the system. Our work is based on a new recursive decomposition theorem for \mbox{SPM} configurations using the notion of staircase bases. Based on this theorem, we provide a recursive formula for the enumeration of $\mbox{SPM}(n)$ and a constant amortized time (CAT) algorithm for the generation of all $\mbox{SPM}(n)$ configurations. The extension of the same approach to the Ice Pile Model is also discussed.
\end{abstract}


\section{Introduction}

The Sand Pile Model (\mbox{SPM}) is a discrete dynamic model inspired by real world physics problems, namely the dynamic of piles of granular materials such as sand or cereals in silos. A first discrete dynamic system formulation of sand piles was given by statistical physicists. The Sand Pile Model we consider here is a one-dimensional simplified model, where sand grains are stacked on a number of adjacent columns.

We now give a precise definition of \mbox{SPM}. A \emph{partition} is an infinite non-increasing sequence $(s_i)_{i \geq 0}$ of natural numbers with finite support. We denote by $\mbox{Part}(n)$ the set of partitions of $n$, \textit{i.e.} partitions $(s_i)_{i \geq 0}$ with $\sum_{i \geq 0} s_i = n$. We notice that we index the components of our sequences starting with zero. A \emph{configuration} is simply a partition. We define the set of \mbox{SPM} configurations with $n$ grains, denoted by $\mbox{SPM}(n)$, to be the set of configurations reachable from the initial configuration $(n,0, \ldots)$ with the following evolution rule called the $\mbox{FALL}$ rule:
\[ s = (s_0, \ldots, s_l, s_{l+1}, \ldots) \to s' = (s_0, \ldots, s_{l}-1, s_{l+1}+1, \ldots) \]
whenever $s_{l} \geq s_{l+1}+2$. Clearly, $\mbox{SPM}(n)$ is a subset of $\mbox{Part}(n)$.

The Ice Pile Model (\mbox{IPM}) is an extension of \mbox{SPM} with the following additional rule:
\[ (s_0, \ldots, s_l, \ldots, s_{l+k'}, \ldots) \to (s_0, \ldots, s_{l}-1, s_{l+1}, \ldots, s_{l+k'}, s_{l+k'+1}+1, \ldots) \]
whenever $s_{l}-1 = s_{l+1} = \ldots = s_{l+k'} = s_{l+k'+1} + 1$ for $k' < k$. This rule, parametrized by the integer $k$, is called the $\mbox{SLIDE}_{k}$ rule. We define the set of $\mbox{IPM}_{k}$ configurations with $n$ grains, denoted by $\mbox{IPM}_{k}(n)$, to be the set of configurations reachable from the initial configuration $(n, 0, \ldots)$ by applications of the rules $\mbox{FALL}$ and $\mbox{SLIDE}_{k}$. We also have that $\mbox{IPM}_{k}(n)$ is a subset of $\mbox{Part}(n)$.

Some results in counting configurations in $\mbox{SPM}(n)$ and $\mbox{IPM}_{k}(n)$ are already known. For instance, in \cite{latapy2001structure}, recursive formulae for $|\mbox{SPM}(n)|$ are given based on an inductive lattice structure of $\mbox{SPM}(n)$, and in \cite{corteel2002enumeration} the generating function of $|\mbox{IPM}_{k}(n)|$ is studied and its asymptotic behavior is given. On the front of exhaustive generation, recent results (\textit{e.g.} \cite{massazza2008cat, massazza2010ipm}) provide efficient exhaustive generation algorithm, more precisely in constant amortized time (CAT), using the dynamics in the evolution of sandpile configurations. However, these two lines of research rely on different aspects of \mbox{SPM}, unlike other combinatorial structures, whose counting and exhaustive generation are often the two sides of the same coin. In this article, we would like to provide a general framework in a combinatorial perspective for both counting, and efficient exhausitive generation of \mbox{SPM}/\mbox{IPM} configurations.

In this article, we study a recursive structure of \mbox{SPM} configurations that determines a recursive decomposition for them. This recursive decomposition is the key to our results on a new recursive formula for $|\mbox{SPM}(n)|$ and a new CAT algorithm to enumerate all configurations in $\mbox{SPM}(n)$. In Section 2, the basic notion of ``staircase basis'' for \mbox{SPM} is introduced, and with this notion we characterize the aforementioned combinatorial recursive structure of accessible configurations in $\mbox{SPM}(n)$, and we obtain in particular a recursive formula to determine their number. We also present a natural algorithm for the exhaustive enumeration of $\mbox{SPM}(n)$ using the recursive structure of accessible configurations, and we prove it to be CAT. Our algorithm differs from the Massazza-Radicioni algorithm because we only use  combinatorial properties of the accessible configurations and not properties related to the dynamic of the system. In Section 3, we give an intuitive presentation of how our idea can be generalized to \mbox{IPM}. We conclude with some discussions of possible directions of future work.

\section{Staircase bases and recursive structure}

\subsection{Staircase bases} \label{sect:staircase-bases}
Let $L_B(n)$ be the lattice obtained by equipping $\mbox{Part}(n)$, the set of partitions, with the ``dominance'' order $\preceq$ defined as follows: for $s,t \in Part(n)$, $s \prec t \iff \forall j, \sum_{i=0}^{j} s_i \geq \sum_{i=0}^{j} t_i$. From \cite{goles1993games} we know that $\mbox{SPM}(n)$ is a sublattice of $L_B(n)$. We can also define the following partial order called \emph{sequence order} or \emph{covering order} on $\mbox{Part}(n)$: for $s,t \in Part(n)$, we write $s \leq t$ if and only if for all $i \in \mathbb{N}$, we have $s_i \leq t_i$. This order can be readily generalized to the set of arbitrary sequences of integers.

Now we will introduce our notion of staircase basis for \mbox{SPM}.

\begin{defi}{\textbf{Staircase bases for \mbox{SPM}}} \label{def:base}
For $k \in \mathbb{N}$, we define the \emph{staircase} of order $k$ by $s(k)=(k, k-1, \ldots, 2, 1, 0, \ldots) \in L_{B}$. More precisely, $\forall i \in \mathbb{N}, s(k)_i=\max(0,k-i)$. We define $B=\{s(k) \mid k \in \mathbb{N}\}$ to be the set of \emph{staircase bases for \mbox{SPM}}.
\end{defi}

We have the following property that follows immediately from the definition of $\mbox{FALL}$.

\begin{prop} \label{prop:base-no-fall}
For $s \in B$, $s$ is a fixed point for \mbox{SPM}, that is, we cannot apply $\mbox{FALL}$ on any column of $s$.
\end{prop}

We now define a parameter of \mbox{SPM} configurations, called \emph{staircase width}, which will be crucial in the following.

\begin{defi}
For some $t \in \mbox{SPM}(n)$, we define its \emph{staircase width} as the integer $sw(t) = \max_{s(k) \leq t}k$. Furthermore, if $sw(t) = w$,  we call $s(w)$ the \emph{staircase socle} (or simply \emph{socle}) of $t$.
\end{defi}

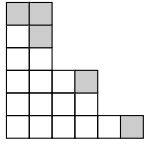
\begin{figure}[!htbp]
\centering
\begin{tikzpicture}
\def \mysqr{rectangle +(0.3,0.3)}
\foreach \x in {0,0.3,0.6,0.9,1.2} \draw (0,\x) \mysqr;
\foreach \x in {0,0.3,0.6,0.9} \draw (0.3,\x) \mysqr;
\foreach \x in {0,0.3,0.6} \draw (0.6,\x) \mysqr;
\foreach \x in {0,0.3} \draw (0.9,\x) \mysqr;
\draw (1.2,0) \mysqr;
\filldraw[black!20] (0,1.5) \mysqr;
\draw (0,1.5) \mysqr;
\filldraw[black!20] (0.3,1.2) \mysqr;
\draw (0.3,1.2) \mysqr;
\filldraw[black!20] (0.3,1.5) \mysqr;
\draw (0.3,1.5) \mysqr;
\filldraw[black!20] (0.9,0.6) \mysqr;
\draw (0.9,0.6) \mysqr;
\filldraw[black!20] (1.5,0) \mysqr;
\draw (1.5,0) \mysqr;
\end{tikzpicture}\caption{$t = (6,6,3,3,1,1,0)$ with  $sw(t)=5$ and with its socle in white} \label{figure:socle}
\end{figure}

The staircase width of an \mbox{SPM} configuration is ``monotone'' with respect to the evolution rule $\mbox{FALL}$, as showed by the following theorem which relies esssentially on Proposition \ref{prop:base-no-fall}.

\begin{prop} \label{thm:bw-monotone}
For $a,b \in \mbox{SPM}(n)$ such that $a \to b$, we have $sw(a) \leq sw(b)$. More generally, if $a \preceq b$, then $sw(a) \leq sw(b)$.
\end{prop}

\begin{proof}
We deal with the case $a \to b$, the general case follows as a consequence. Suppose we had $sw(a) > sw(b)$. In this case, there exists some index $i$ such that $a_i \geq s(sw(a))_i$ but $b_i < s(sw(a))_i$. This is possible only when we apply $\mbox{FALL}$ on $a$ at  index $i$, thus $a_i = b_i + 1$ and $a_i = s(sw(a))_i$. However, $a_{i+1} \geq s(sw(a))_{i+1}$. By Proposition \ref{prop:base-no-fall}, $\mbox{FALL}$ cannot be applied at  index $i$, which is a contradiction. Therefore $sw(a) \leq sw(b)$. 
\end{proof}

In \cite{goles1993games}, the following characterization of the (unique) fixed point in $\mbox{SPM}(n)$ is given.

\begin{prop} \label{thm:fixpoint}
In $\mbox{SPM}(n)$, the unique fixed point with respect to rule $\mbox{FALL}$ is
\[ \phi(n) = (k, k-1, \ldots, l+1, l, l, l-1, \ldots, 2, 1, 0, \ldots) \]
where $(k,l)$ is the unique pair such that $0 \leq l \leq k$ and  $n=\frac{1}{2}k(k+1)+l$.
\end{prop}

As a corollary of Propositions above, we obtain an upper bound for the staircase width of elements in $\mbox{SPM}(n)$.

\begin{prop} \label{prop:bw-upperbound}
For all $t \in \mbox{SPM}(n)$, we have $sw(t) \leq \sqrt{2n}$.
\end{prop}

\begin{proof}
By Proposition \ref{thm:bw-monotone} and \ref{thm:fixpoint}, $sw(t) \leq sw(\phi(n)) = k \leq \sqrt{2n}$.
\end{proof}

We can partition the set $\mbox{SPM}(n)$ according to the staircase width of configurations. We define $\mbox{SPM}(n,w)=\{s \in \mbox{SPM}(n) | sw(s)=w\}$ be the subset of $\mbox{SPM}(n)$ of all elements with staircase width $w$. For $w$ running from $1$ to $\lfloor \sqrt{2n} \rfloor$, all $\mbox{SPM}(n,w)$ partition $\mbox{SPM}(n)$. From now on, we concentrate on $\mbox{SPM}(n,w)$ instead of $\mbox{SPM}(n)$ as a whole. To generate $\mbox{SPM}(n)$ exhaustively, it suffices to provide a CAT algorithm to generate elements of $\mbox{SPM}(n,w)$, with the parameter $w$ varying from minimal value 1 to maximal value less than $\sqrt{2n}$.

\subsection{Recursive structure} \label{sect:rec-struct}

In \cite{goles2002sandpiles}, the following characterization of elements in $\mbox{SPM}(n)$ is given.

\begin{thm} \label{thm:spm-chara}
A partition $s$ is in $\mbox{SPM}(n)$ if and only if none of the following patterns (also called forbidden patterns) occur:
\begin{itemize}
\item $p,p,p$ for $p > 0$ (that is, three columns containing the same number of grains)

\item $p,p,p-1,p-2, \ldots, q+2,q+1,q,q$ for $p>q>0$ (that is, two plateaux, one of height $p$ and one of height $q$, separated by a perfect staircase)
\end{itemize}
\end{thm}

Using this theorem, we can bound the number of non-zero components of an element in $\mbox{SPM}(n,w)$. Given two finite sequences $a,b$, we denote their concatenation by $a \cdot b$.

\begin{lem} \label{lem:bw-seq-order}
The largest (in the sequence order $\leq$) $\mbox{SPM}$ configuration $s$ with $s_0 \leq w$ is $w \cdot s(w)$.
\end{lem}
\begin{proof}
Let $s$ be an $\mbox{SPM}$ configuration with $s_0 \leq w$. Theorem~\ref{thm:spm-chara} implies that, for any integers $i \geq 0$ and $k > 0$, $s_{i+k} \leq s_{i}-k+1$, thus $s_{k} \leq w-k+1$ and $s \leq w \cdot s(w)$. We conclude by noticing that $w \cdot s(w)$ is also an $\mbox{SPM}$ configuration.
\end{proof}

\begin{prop} \label{prop:bw-width}
For $a \in \mbox{SPM}(n,w)$, we have $a_{w+1}=0$.
\end{prop}

\begin{proof}
For $a \in \mbox{SPM}(n,w)$, we have $sw(a)=w$, thus there exists an index $i \leq w$ such that $a_i=w-i$. The suffix $(a_i, a_{i+1}, \ldots)$ is also an $\mbox{SPM}$ configuration. By Lemma~\ref{lem:bw-seq-order}, we have $a_{w+1} \leq 0$.
\end{proof}

As a consequence, we can express an element $a \in \mbox{SPM}(n,w)$ as a $(w+1)$-tuple $a=(a_0,\ldots,a_{w})$. We now introduce a representation of elements in $\mbox{SPM}(n,w)$ obtained by ``removing'' the socle of each element.

\begin{defi}{\textbf{Reduced form for} $\mbox{SPM}(n,w)$} \label{def:reducedform}
For $s \in \mbox{SPM}(n,w)$, we define $red_w(s)$ as follows:
\[ \forall i \in \{0, \ldots, w\}, (red_w(s))_i = s_i - w + i. \]
Thus $red_w(s)$ is obtained by simply substracting $(w, w-1, \ldots, 0)$ pointwise from $s$. We call $red_w(s)$ the \emph{reduced form} of $s$.
\end{defi}

Note that, in the above definition, the subscript $w$ is not necessary. It has been added for emphasizing the staircase width of the original configuration, as well as the number of components of the corresponding reduced form. Components in a reduced form may be zero. We notice that the leftmost component that is equal to zero may be the component in position $w$. Think for instance of configuration $(4,3)$ whose reduced form is $(2,2,0)$. For our proposes, however, it is important to highlight this apparently superflous component, as the leftmost zero component of a reduced form is pivotal to define our recursive decomposition. 

Let $R(n,w)$ be the set of reduced forms of elements in $\mbox{SPM}(n,w)$. The map $red_w$ is clearly a bijection between $R(n,w)$ and $\mbox{SPM}(n,w)$. We should also notice that a reduced form is not necessarily an \mbox{SPM} configuration. An example of an \mbox{SPM} configuration and of its reduced form is given in Figure \ref{figure:decomp}, where the white part is the socle and the gray part corresponds to the reduced form.

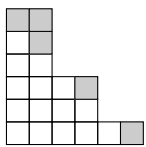
\begin{figure}[!htbp]
\centering
\begin{tikzpicture}
\def \mysqr{rectangle +(0.3,0.3)}
\foreach \x in {0,0.3,0.6,0.9,1.2} \draw (0,\x) \mysqr;
\foreach \x in {0,0.3,0.6,0.9} \draw (0.3,\x) \mysqr;
\foreach \x in {0,0.3,0.6} \draw (0.6,\x) \mysqr;
\foreach \x in {0,0.3} \draw (0.9,\x) \mysqr;
\draw (1.2,0) \mysqr;
\filldraw[black!20] (0,1.5) \mysqr;
\draw (0,1.5) \mysqr;
\filldraw[black!20] (0.3,1.2) \mysqr;
\draw (0.3,1.2) \mysqr;
\filldraw[black!20] (0.3,1.5) \mysqr;
\draw (0.3,1.5) \mysqr;
\filldraw[black!20] (0.9,0.6) \mysqr;
\draw (0.9,0.6) \mysqr;
\filldraw[black!20] (1.5,0) \mysqr;
\draw (1.5,0) \mysqr;
\node[text width = 7cm, text centered] at (0.9,-0.55) {$t = (6,6,3,3,1,1,0)$ with $r = red_5(t) = (1,2,0,1,0,1)$};
\end{tikzpicture}
\caption{An example of \mbox{SPM} configuration and of its reduced form} \label{figure:decomp}
\end{figure}

The characterization of elements in $\mbox{SPM}(n,w)$ translates into a characterization of their reduced forms.

\begin{prop} \label{prop:reduceform-chara}
A $(w+1)$-tuple $r = (r_0, r_1, \ldots, r_w)$ of natural numbers is in $R(n,w)$ if and only if the following conditions are satisfied.
\begin{itemize}
\item [\textbf{(i)}] $\sum_{i=0}^{w} r_i = n - \frac{1}{2}w(w+1)$.
\item [\textbf{(ii)}] There exists at least an index $i_0$ with $0 \leq i_0 \leq w$ such that $r_{i_0}=0$.
\item [\textbf{(iii)}] For all $i,j$ such that $0 \leq i < j \leq w$, we have $r_i \geq r_j - 1$.
\end{itemize}
\end{prop}

\begin{proof}
Let $t \in \mbox{SPM}(n,w)$ and $r = red_{w}(t)$. Conditions (i) and (ii) follow directly from the definition of $\mbox{SPM}(n,w)$ and of $red_{w}$. For condition (iii), we notice that Theorem~\ref{thm:spm-chara} implies that $t_{i+k} \leq t_i - k + 1$ for any $k>0$ and $i \geq 0$, thus for any $i,j$ with $0 \leq i < j \leq w$, $r_j + w - j = t_j \leq t_i - (j - i) + 1 = r_i + w - j + 1$, and we have $r_j \leq r_i + 1$.

For the other direction, let $r$ be a $(w+1)$-tuple that satisfies (i), (ii) and (iii), and let $t$ be the $(w+1)$-tuple obtained by adding $r$ and $s(w)$ component-wise. To prove that $r \in R(n,w)$, it suffices to prove that $t \in \mbox{SPM}(n,w)$. From condition (iii), it follows that $t_i = r_i + w - i \geq r_{i+1} - 1 + w - i = t_{i+1}$, thus $t$ is a partition. By Theorem~\ref{thm:spm-chara}, we only need to prove that no forbidden pattern exists in $t$. We suppose that such a pattern exists between column $i$ and $j$, we then have $t_i - (j - i) + 1 < t_j$, thus $r_i + w - j + 1 < r_j + w - j$, which implies $r_i < r_j - 1$, contradicting condition (iii). Hence $t \in \mbox{SPM}(n,w)$.
\end{proof}

We notice that condition (i) only ensures the weight $n$ to be correct. Therefore, any $(w+1)$-tuple $r$ verifying (ii) and (iii) must be in some $R(n,w)$ for an appropriate $n$. We also notice that condition (iii) is preserved by taking prefixes. That is to say, if a tuple $r$ verifies (iii), then all of its prefixes also verify (iii). With these remarks, we provide the following decomposition theorem for reduced forms, which is the main result of this article. Both our enumeration formula and our exhaustive generation algorithm rely on this theorem.

\begin{thm}{\textbf{Decomposition of reduced forms}} \label{thm:rec-decomp-red}

A reduced form $r \in R(n,w)$ can be uniquely decomposed into the following form:
\[ (r_0, \ldots, r_w) = (t_0, \ldots, t_{l-1}, 0, u_0, \ldots, u_{w-l-1}), \]
such that $t = (t_0, \ldots, t_{l-1})$ and $u=(u_0, \ldots, u_{w-l-1})$ verify the following conditions.
\begin{itemize}
\item If $u$ is not empty, for all $i \in \{0, \ldots, w - l - 1\}$, we have $u_i \in \{0, 1\}$.
\item If $t$ is not empty, for all $i \in \{0, \ldots, l - 1\}$, we have $t_i > 0$.
\item In the case that $t$ is not empty, let $m$ be the minimum of $t_i$ for $0 \leq i \leq l-1$ and $r' = (t_0-m, t_1-m, \ldots, t_{l-1}-m)$. Then the tuple $r'$ is in some $R(n',l-1)$, where $n'$ can be easily calculated.
\end{itemize}
We refer to this decomposition by writing $r = (l,u,m), r'$. We represent an empty tuple by a pair of parentheses $()$. When $t$ is empty, we take $m=0$.
\end{thm}

\begin{proof}
We start by the validity of our decomposition. As $r \in R(n,w)$, by condition (ii) in Proposition \ref{prop:reduceform-chara}, there exists some index $i$ such that $r_{i}=0$. To ensure all parts of $t$ to be strictly positive, we take $l$ to be the minimum $i_0$ such that $r_{i_0}=0$. By condition (iii) in Proposition \ref{prop:reduceform-chara}, we deduce that $u$ is a sequence of $0$'s and $1$'s. The tuple $r'$ clearly has only positive parts and has at least one $0$ component, thus it verifies (ii). The fact that $r'$ satisfies condition (iii) is provided by the fact that property (iii) is invariant not only by taking prefixes (so $t$ satisfies it), but also when a constant is substracted from every part of a partition (and $r'$ is obtained from $t$ by subtracting the integer $m$ from all its parts). Therefore $r'$ is in some $R(n', l-1)$. The weight $n'$ can be easily calculated from $m$ and $u$.

Our decomposition is clearly unique by definition.
\end{proof}

As an example, we consider the decomposition of the reduced form $r = (1,2,0,1,0,1)$ in Figure \ref{figure:decomp}. Clearly we have $r = (2,(1,0,1),1), (0,1)$.

Since $r'$ in the decomposition is also a reduced form, we can apply the decomposition recursively. An example of a full recursive decomposition can be found in Figure \ref{figure:decomp4}. Different colors are used for different levels of decomposition. Graphically, we can see that at each level, we always have a skewed strip of thickness $m$ on the left of position $l$, and some ``dust grains'' corresponding to $u$ on the right of position $l$.

\begin{figure}[!htbp]
\centering
\begin{tikzpicture}
\def \mysqr{rectangle +(0.3,0.3)}
\foreach \y in {0,0.3,0.6,0.9,1.2,1.5,1.8,2.1,2.4,2.7,3.0,3.3} \draw (0,\y) \mysqr;
\foreach \y in {0,0.3,0.6,0.9,1.2,1.5,1.8,2.1,2.4,2.7,3.0} \draw (0.3,\y) \mysqr;
\foreach \y in {0,0.3,0.6,0.9,1.2,1.5,1.8,2.1,2.4,2.7} \draw (0.6,\y) \mysqr;
\foreach \y in {0,0.3,0.6,0.9,1.2,1.5,1.8,2.1,2.4} \draw (0.9,\y) \mysqr;
\foreach \y in {0,0.3,0.6,0.9,1.2,1.5,1.8,2.1} \draw (1.2,\y) \mysqr;
\foreach \y in {0,0.3,0.6,0.9,1.2,1.5,1.8} \draw (1.5,\y) \mysqr;
\foreach \y in {0,0.3,0.6,0.9,1.2,1.5} \draw (1.8,\y) \mysqr;
\foreach \y in {0,0.3,0.6,0.9,1.2} \draw (2.1,\y) \mysqr;
\foreach \y in {0,0.3,0.6,0.9} \draw (2.4,\y) \mysqr;
\foreach \y in {0,0.3,0.6} \draw (2.7,\y) \mysqr;
\foreach \y in {0,0.3} \draw (3,\y) \mysqr;
\foreach \y in {0} \draw (3.3,\y) \mysqr;
\filldraw[black!20] (3.6,0) \mysqr;
\filldraw[black!20] (3.3,0.3) \mysqr;
\filldraw[black!20] (2.4,1.5) \mysqr;
\filldraw[black!20] (2.4,1.2) \mysqr;
\filldraw[black!20] (2.1,1.8) \mysqr;
\filldraw[black!20] (2.1,1.5) \mysqr;
\filldraw[black!20] (1.8,2.1) \mysqr;
\filldraw[black!20] (1.8,1.8) \mysqr;
\filldraw[black!20] (1.5,2.4) \mysqr;
\filldraw[black!20] (1.5,2.1) \mysqr;
\filldraw[black!20] (1.2,2.7) \mysqr;
\filldraw[black!20] (1.2,2.4) \mysqr;
\filldraw[black!20] (0.9,3) \mysqr;
\filldraw[black!20] (0.9,2.7) \mysqr;
\filldraw[black!20] (0.6,3.3) \mysqr;
\filldraw[black!20] (0.6,3) \mysqr;
\filldraw[black!20] (0.3,3.6) \mysqr;
\filldraw[black!20] (0.3,3.3) \mysqr;
\filldraw[black!20] (0,3.9) \mysqr;
\filldraw[black!20] (0,3.6) \mysqr;
\filldraw[black!50] (2.4,1.8) \mysqr;
\filldraw[black!50] (1.8,2.4) \mysqr;
\filldraw[black!50] (1.5,2.7) \mysqr;
\filldraw[black!50] (0.6,3.6) \mysqr;
\filldraw[black!50] (0.3,3.9) \mysqr;
\filldraw[black!90] (0.3,4.2) \mysqr;
\filldraw[black!50] (0,4.2) \mysqr;
\filldraw[black!90] (0,4.5) \mysqr;
\draw (3.3,0.3) \mysqr;
\draw (3.6,0) \mysqr;
\draw (2.4,1.2) \mysqr;
\draw (2.4,1.5) \mysqr;
\draw (2.1,1.5) \mysqr;
\draw (2.1,1.8) \mysqr;
\draw (1.8,1.8) \mysqr;
\draw (1.8,2.1) \mysqr;
\draw (1.5,2.1) \mysqr;
\draw (1.5,2.4) \mysqr;
\draw (1.2,2.4) \mysqr;
\draw (1.2,2.7) \mysqr;
\draw (0.9,2.7) \mysqr;
\draw (0.9,3) \mysqr;
\draw (0.6,3) \mysqr;
\draw (0.6,3.3) \mysqr;
\draw (0.3,3.3) \mysqr;
\draw (0.3,3.6) \mysqr;
\draw (0,3.6) \mysqr;
\draw (0,3.9) \mysqr;
\draw (2.4,1.8) \mysqr;
\draw (1.8,2.4) \mysqr;
\draw (1.5,2.7) \mysqr;
\draw (0.6,3.6) \mysqr;
\draw (0.3,3.9) \mysqr;
\draw (0.3,4.2) \mysqr;
\draw (0,4.2) \mysqr;
\draw (0,4.5) \mysqr;
\node[text width = 8cm, text centered] at (1.8,-0.35) {$r = (4,4,3,2,2,3,3,2,3,0,0,1,1)$};
\end{tikzpicture}
\caption{An example of full recursive decomposition of reduced form} \label{figure:decomp4}
\end{figure}
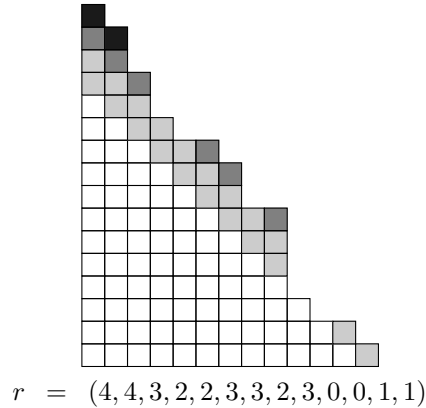

\subsection{Construction of generating sequences}
Before proceeding to counting and generating $\mbox{SPM}(n)$, we provide a construction of a sequence of applications of the  $\mbox{FALL}$ rule that allows to obtain any given accessible \mbox{SPM} configuration from the initial one. This construction is essentially an alternative proof of one direction of Theorem~\ref{thm:spm-chara}.

\begin{defi}{(\textbf{Generating Sequence})}
For $t \in \mbox{SPM}(n)$, we say that a finite sequence $(a_i) = a_0, a_1, \ldots, a_{l-1}$ \emph{generates} $t$ if $t$ can be obtained from $(n, 0, \ldots)$ by successively applying $\mbox{FALL}$ at index $a_i$ for $i$ from $0$ to $l-1$, which implies in particular that each such application must be valid. In this case, we call $(a_i)$ a \emph{generating sequence} of $t$.
\end{defi}

We can say that a generating sequence $(a_i)$ is a certificate that $t$ is in $\mbox{SPM}(n)$, as it provides a path to go from $(n, 0, \ldots)$ to $t$ by applying $\mbox{FALL}$. Given a finite sequence $(a_i)$, we can check its validity by applying $\mbox{FALL}$ accordingly, and if it is indeed valid, we also obtain the corresponding $t$. Conversely, we will now provide a method that, given $t \in \mbox{SPM}(n)$, constructs a generating sequence $(a_i)$ of $t$. We also denote by $a^{[m]}$ the $m$-fold repetition of $a$. We recall that the concatenation operation of two sequences is denoted by $\cdot$

First we will construct a generating sequence for staircases $s(k)$. Let $\alpha_i$ be the sequence $0,1,\ldots,i-1$ and $\beta_i$ be the sequence $\alpha_i \cdot \alpha_{i-1} \cdot \ldots \cdot \alpha_1$ (here $/cdot$ denotes the concatenation of sequences). Let $\widehat{s(k)}$ be the partition obtained by adding $k+1$ grains to the first column of $s(k)$. It can be verified that $\beta_k$ applied to $\widehat{s(k)}$ produces $s(k+1)$. With this observation, we see clearly that $\beta_1 \cdot \beta_2 \ldots \beta_{k-1}$ generates $s(k)$ when we start from the initial configuration $(n, 0, 0, \ldots)$ with $n=k(k+1)/2$. For $n \geq k(k+1)/2$, the same sequence gives the configuration $(n-k(k-1)/2, k-1, k-2, \ldots, 1, 0)$, corresponding to the reduced form $(n-k(k+1)/2, 0, \ldots,0)$.

We want to study the effects of the rule $\mbox{FALL}$ on reduced forms. We define the rule $\mbox{FALL}'$ operating on reduced forms $r \in R(n,w)$ as follows:
\[ r= (r_0, \ldots, r_{w}) \to r' = (r_0, \ldots, r_{l}-1, r_{l+1}+1, \ldots, r_{w}), \]
if $r_{l} \geq r_{l+1}+1$.

This is equivalent to say that the following diagram commutes:

\begin{center}
\begin{tikzpicture}
\node (s) at (0,2) {$t$};
\node (t) at (2,2) {$t'$};
\node (sred) at (0,0) {$r$};
\node (tred) at (2,0) {$r'$};
\draw[->] (s.south) -- (sred.north) node[left,midway] {$red_w$};
\draw[->] (t.south) -- (tred.north) node[right,midway] {$red_w$};
\draw[->] (s.east) -- (t.west) node[above,midway] {$\mbox{FALL}(l)$};
\draw[->] (sred.east) -- (tred.west) node[above,midway] {$\mbox{FALL}'(l)$};
\end{tikzpicture}
\end{center}

To construct $t \in \mbox{SPM}(n)$, we denote by $w=sw(s)$ its staircase width and $r=red_w(t)$ its reduced form. First we use $\beta_1 \cdot \beta_2 \cdot \ldots \cdot \beta_{w-1}$ to construct the  socle of $t$, then we pick up the viewpoint of reduced form to construct the rest. The remaining task is then to construct a path from the reduced form $(n-w(w+1)/2, 0, \ldots,0)$ to $r$ using $\mbox{FALL}'$.

For simplicity, we denote by $(0)$ the only element in $R(0,0)$. Using Theorem \ref{thm:rec-decomp-red}, we now define recursively a function $Path_{n,w}$ such that, given a reduced form $r \in R(n,w)$, constructs a path from $(n-w(w+1)/2, 0, \ldots,0)$ to $r$ in the Hasse diagram of $\mbox{SPM}(n)$ by applying $\mbox{FALL}'$. 

For $w = 0$, $Path_{n,w}(r)=()$, the empty sequence. For $w>0$, let $r=((t,m),l,u)$ be the decomposition of $r$. For $i_1 > \ldots > i_k > l$ such that $u_{i_j - l - 1}=1$ for $1 \leq j \leq k$, we define $seq_0(u)=\alpha_{i_1} \cdot \alpha_{i_2} \cdot \ldots \cdot \alpha_{i_k}$. Each $\alpha_{i_j}$ sends a grain (\textit{i.e.} adds 1) to the component $r_{i_j}=u_{i_j-l-1}=1$. In order to construct the $m$ layers of grains from position $0$ to position $l-1$, we only need to repeat $m$ times the sequence $\alpha_{l-1} \cdot \alpha_{l-2} \ldots \alpha_{1}$. We define $seq_1(l,m)=(\alpha_{l-1} \cdot \alpha_{l-2} \cdot \ldots \cdot \alpha_{1})^{m}$. Finally we define $Path_{n,w}(r)$ recursively as follows:
\[ Path_{n,w}(r)=seq_0(u) \cdot seq_1(l,m) \cdot Path_{n-k-lm,l-1}(t). \]

We have the following proposition stating the correctness of the construction $Path_{n,w}$.

\begin{prop} \label{prop:genseq}
For $s \in \mbox{SPM}(n)$ with $w=sw(s)$, $\beta_{1} \cdot \beta_{2} \cdot \ldots \cdot \beta_{w-1} \cdot Path_{n,w}(red_w(s))$ is a generating sequence of $s$.
\end{prop}

\begin{proof}
We perform an induction on $w$. The base case $w = 0$ is trivial. To proceed by induction, we suppose that the proposition is true for any $s \in \mbox{SPM}(n)$ having width smaller than $w$. By combining  the notation in this proposition and the result of Theorem \ref{thm:rec-decomp-red}, we have $red_w(s) = (l, u, m), r'$. We observe that $r'$ is a reduced form of a certain configuration $s'$ with width $l < w$. It is easy to verify that the sequence $Path_{n,w}$ (by its recursive definition and our induction hypothesis) constructs the correct reduced form, which completes construction of the socle obtained by the successive applications of the $\beta_i$'s.
\end{proof}

This proposition can be seen as a constructive proof of one direction of Theorem \ref{thm:spm-chara}. We should be aware that this construction is not unique. For instance, in the recursive definition of $Path_{n,w}$, by exchanging the order of $seq_0(u)$ and $seq_1(l,m)$, we can get a different valid construction. From a physical point of view, it may be interesting to study the number of generating sequences of configurations in order to have better understanding over the generic evolution of an \mbox{SPM} model.

\subsection{Recursive formula for $|\mbox{SPM}(n)|$} \label{sect:counting-spm}

The recursive structure of $\mbox{SPM}(n)$ described by Theorem \ref{thm:rec-decomp-red} can be used to give a counting formula for $\mbox{SPM}(n)$.

We define $c(p,w) = |R(p+w(w+1)/2,w)|$. The following proposition follows directly from the definitions of $c(p,w)$ and $R(p,w)$.

\begin{prop} \label{prop:spm-as-sum}
For a natural number $n$, we have
\[ |\mbox{SPM}(n)| = \sum_{\substack{w \geq 1 \\ w(w+1) \leq 2n}} c\left( n-\frac{w(w+1)}{2},w \right) . \]
\end{prop}

The reason we choose $c(p,w) = |R(p+w(w+1)/2,w)|$ is that for all $r \in R(p+w(w+1)/2,w)$, we have $\sum_i r_i = p$. Here $p$ represents the number of grains located ``above'' the socle $s(w)=(w,w-1,\ldots,1,0)$. As a consequence of Theorem \ref{thm:rec-decomp-red}, we have the following recurrence for $c(p,w)$.

\begin{thm} \label{thm:rec-count-red}
The value of $c(p,w)$ is uniquely determined by the following recurrence. For $w \geq 0$, we have $c(0,w)=1$. For $p \neq 0$, we have $c(p,0)=0$. For the remaining cases,
\[ c(p,w) =  \binom{w}{p} + \sum_{l=1}^{w} \sum_{i=0}^{\min(w-l,p-l)} \sum_{m=1}^{\lfloor \frac{p-i}{l} \rfloor} \bigg[ \binom{w-l}{i}c(p-i-lm, l-1) \bigg]. \]
\end{thm}

\begin{proof}
This recurrence comes directly from the decomposition of reduced forms $r=((t,m),l,u)$. The base cases of the recurrence can be easily verified. The summation index $l$ (resp. $m$) corresponds to the integer $l$ (resp. $m$) in the decomposition. The summation index $i$ stands for the number of parts of $u$ equal to $1$. Binomial coefficients arise by taking the number of all possible sequences $u$ of $0$'s and $1$'s having exactly $i$ parts equal to $1$ (and all other parts equal to $0$, as in Theorem \ref{thm:rec-decomp-red}). The special case $l=0$ is treated in the first term on the right hand side. This is the case where $t$ is empty. The $p$ grains in $r$ can be only placed on $w$ columns in $u$, and at most one grain can be placed on each column. Therefore, the first term is non zero only when $0 \leq p \leq w$.
\end{proof}

Now we want to evaluate the complexity of computing $|\mbox{SPM}(n)|$ using this recursive formula.

All binomial coefficients $\binom{a}{b}$ needed to calculate $c(n,w)$ verify $a \leq w$. By Proposition \ref{prop:bw-upperbound}, we have $w \leq \sqrt{2n}$. Therefore, by memorizing results (in the manner of dynamic programming), we can precalculate all of them using $O(n)$ additions.

In the recurrence for $c(p,w)$, we notice that $m$ cannot exceed $p/l$, therefore for a fixed $l \leq w$, there are at most $p/l$ possibilities for the value of $m$, thus we have $\sum_{l=1}^{w} p/l = p \sum_{l=1}^{w} 1/l = O(p\log(w))$ possible pairs $(l,m)$. Since $1 \leq i \leq w$, in the recurrence for $c(p,w)$ there are at most $O(wp\log(w))$ terms, thus $O(wp\log(w))$ arithmetic operations are needed to calculate each $c(p,w)$, given the value of all $c(p', w')$ with $p' < p$ and $w' < w$. The total number of arithmetic operations for calculating all $c(p,w)$, for all $p \leq n$, is bounded by:

\[ \sum_{p \leq n} \sum_{1 \leq w \leq \sqrt{2n}} wp\log(w) = O(n^{3}\log(n)) \]

It follows from Proposition~\ref{prop:spm-as-sum} that we only need $O(n^{3}\log(n))$ arithmetic operations to compute $|\mbox{SPM}(n)|$. According to \cite{corteel2002enumeration}, we can bound $|\mbox{SPM}(n)|$ by $c^n$ (for a certain constant $c$), thus all coefficients involved have $O(n)$ bits, thus we know that we can compute $|\mbox{SPM}(n)|$ in $O(n^{4} \log^2 n \log \log n)$ time using fast integer multiplication.

As a remark, given the recursive formula for $c(p,w)$, it is straightforward to construct a uniform random generator of $\mbox{SPM}(n)$ by computing all $c(p,w)$ and generate configurations recursively in a uniformly random way using appropriate probabilities computed with $c(p,w)$.

\subsection{A CAT algorithm for $\mbox{SPM}(n)$}

It is clear that the exhaustive generation of $\mbox{SPM}(n)$ reduces to the exhaustive generation of $\mbox{SPM}(n,w)$ with the staircase width $w$ varying from $1$ to $\lfloor \sqrt{2n} \rfloor$, which in turns reduces to the exhaustive generation of reduced forms in $R(n,w)$. The unique decomposition of reduced forms in Theorem~\ref{thm:rec-decomp-red} thus gives a natural way to exhaustively generate reduced forms in a recursive fashion. Essentially our algorithm will be an algorithmic transcription of Proposition~\ref{prop:spm-as-sum} and of Theorem~\ref{thm:rec-count-red}. Algorithm~\ref{algo:cat-detailed} is an example of such a transcription. It should be called initially with $d=0$. The cases on $l$ are for the further complexity analysis.

\begin{algorithm}[!htb]
\caption{Recursive Generation of $R(p+w(w+1)/2,w)$} \label{algo:cat-detailed}
\textbf{Generate}($p,w,d$)

\KwResult{Each call \textsf{yield()} (including those in recursive calls) sees a new $r \in R(p+w(w+1)/2,w)$ in the array $A$ of triplets $(l_i, u_i, m_i)$. Throughout the code, $l$ is the position of first zero, $i$ the number of 1s in $u$ and $m$ the minimum before the first zero.}
\Begin{
	\If{$p=0$}{
		yield(); \Return{}\;
	}
	\tcp{Case $l = 0$}
	\If{$p \leq w$}{
		\ForEach{$u$ of length $w$ with $p$ 1s}{
			$A[d] \gets (0, 0, u)$; yield()\;
		}
	}
	\tcp{Case $l = 1$}
	\For{$i \gets 0$ \KwTo $\min(w-1,p)$}{
		\ForEach{$u$ of length $w-1$ with $i$ 1s}{
			$A[d] \gets (1, u, p-i)$; yield()\;
		}
	}
	\tcp{Case $l \geq 2$}
	\For{$l \gets 2$ \KwTo $w$}{
		\For{$i \gets 0$ \KwTo $\min(w-l,p-l)$}{
			\ForEach{$u$ of length $w-l$ with $i$ 1s}{
				\For{$m \gets 1$ \KwTo $\lfloor \frac{p-i}{l} \rfloor$}{
					$A[d] \gets (l, u, m)$; Generate($p-i-lm,l-1, d+1$)\;
				}
			}
		}
	}
	\Return{}\;
}
\end{algorithm}

We now explain the data structure we use. From the notation in Theorem~\ref{thm:rec-count-red}, the unique decomposition of reduced form gives each reduced form $r$ an expression as a list of triplets $(l_0, u_0, m_0), \ldots, (l_d, u_d, m_d)$ with natural numbers $l_i \geq 0, m_i > 0$ and $(0,1)$-sequences $u_i$, with the condition that the sequence $(l_i)_{0 \leq i \leq d}$ is strictly decreasing. We will adopt this notation for our exhaustive generation algorithm. In the generation, the natural numbers $l_i, m_i$ come naturally from loop indices, and the rest consists of the generation of $u_i$, which are $(0,1)$-sequences, with given length and given weight (\textit{i.e.} total number of 1s). There are various CAT algorithms for generating $(0,1)$-sequences with given length and weight, for example those presented by Knuth in Chapter 7.2.1.3 of \cite{knuth2005art} and by Ruskey in \cite{ruskey1996combinatorial} in Section 4.3. Any one of these methods can be used as a subroutine to generate $u$. Many such algorithms have a linear initialization when the given length is equal to the given weight, and they might fail to be CAT in this special case. However, this can be fixed by using a boolean variable associated to $u_i$ to indicate this special case. Therefore, we can consider all operations in Algorithm~\ref{algo:cat-detailed} to be performed in constant time.

We now analyse the time complexity of Algorithm~\ref{algo:cat-detailed}. According to Section 4.3 in \cite{ruskey1996combinatorial}, we only need to analyse the form of the recursion tree. For each call of \textsf{Generate(p,w,d)}, it is clear that at least one of the generated $\mbox{SPM}$ configurations  is produced immediately in this call, by putting all the $p$ grains into the first column. Therefore, the number of nodes of \textsf{Generate(p,w,d)} in the recursion tree of \textsf{Generate($n-w(w-1)/2$,$w$,$0$)} is bounded by $|\mbox{SPM}(n,w)|$, thus the total number of nodes is bounded by $2|\mbox{SPM}(n,w)|$. We also know that, in Algorithm~\ref{algo:cat-detailed}, the time spent to spawn a child for each node in the recursion tree is bounded by a constant. It is then immediate that our algorithm is CAT for $\mbox{SPM}(n,w)$, thus also CAT for $\mbox{SPM}(n)$.

We now analyse the space complexity of Algorithm~\ref{algo:cat-detailed}, first expressed in terms of memory cells, then in terms of bits. We notice that generating $u$ of length $k$ uses $O(k)$ extra memory. As the total length of all $u$ in a recursive call is bounded by $w$, we know that this part of memory consumption is $O(w)$. Secondly, $w$ decreases at each recursive call, therefore the recursion depth is $O(w)$. Since \textsf{Generate(p,w,d)} only uses a constant number of scalar variables besides the array $A$, the total stack memory consumption is $O(w)$. Adding the memory needed for the array $A$ and for generating $u$, the total number of memory cells used in \textsf{Generate(p,w,d)} is $O(w)$. As the value of each memory cell is bounded by $\max(n,w)$, space complexity of generating $\mbox{SPM}(n,w)$ is $O(w\log(\max(n,w)))$ bits. For the generation of $\mbox{SPM}(n)$, a simple reuse gives a total space complexity of $O(\sqrt{n}\log(n))$ bits.

\section{Generalization to ice pile model}

We will now generalize previous results to the Ice Pile Models $\mbox{IPM}_k(n)$, using the same terminology as in previous sections.

For $\mbox{IPM}_k(n)$, we define the following staircases for $w>0$ and $1 \leq l \leq k$:
\[ s(w,l) = (\underbrace{w,\ldots,w}_{l}, \underbrace{w-1,\ldots,w-1}_{k}, \ldots, \underbrace{1,\ldots,1}_{k}). \]
In Figure~\ref{fig:bases-ipm} several examples of staircases for $k=2$ are presented.

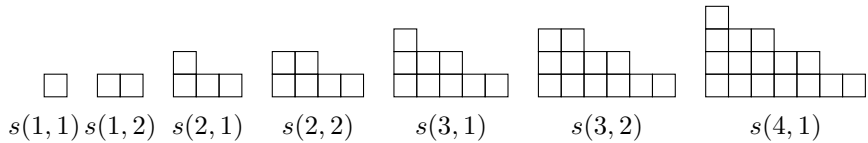
\begin{figure}[!htb]
\centering
\begin{tikzpicture}
\def \mysqr{rectangle +(0.300000, 0.300000)}

\begin{scope}[xshift=0cm]
\draw (0.000,0) \mysqr;
\node at (0.0, -0.4){$s(1,1)$};
\end{scope}

\begin{scope}[xshift=0.7cm]
\draw (0.000,0) \mysqr;
\draw (0.300,0) \mysqr;
\node at (0.3, -0.4){$s(1,2)$};
\end{scope}

\begin{scope}[xshift=1.7cm]
\foreach \y in {0,0.300} \draw (0.000, \y) \mysqr; 
\draw (0.300,0) \mysqr;
\draw (0.600,0) \mysqr;
\node at (0.45, -0.4){$s(2,1)$};
\end{scope}

\begin{scope}[xshift=3cm]
\foreach \y in {0,0.300} \draw (0.000, \y) \mysqr; 
\foreach \y in {0,0.300} \draw (0.300, \y) \mysqr; 
\draw (0.600,0) \mysqr;
\draw (0.900,0) \mysqr;
\node at (0.6, -0.4){$s(2,2)$};
\end{scope}

\begin{scope}[xshift=4.6cm]
\foreach \y in {0,0.300, ..., 0.600} \draw (0.000, \y) \mysqr; 
\foreach \y in {0,0.300} \draw (0.300, \y) \mysqr; 
\foreach \y in {0,0.300} \draw (0.600, \y) \mysqr; 
\draw (0.900,0) \mysqr;
\draw (1.200,0) \mysqr;
\node at (0.75, -0.4){$s(3,1)$};
\end{scope}

\begin{scope}[xshift=6.5cm]
\foreach \y in {0,0.300, ..., 0.600} \draw (0.000, \y) \mysqr; 
\foreach \y in {0,0.300, ..., 0.600} \draw (0.300, \y) \mysqr; 
\foreach \y in {0,0.300} \draw (0.600, \y) \mysqr; 
\foreach \y in {0,0.300} \draw (0.900, \y) \mysqr; 
\draw (1.200,0) \mysqr;
\draw (1.500,0) \mysqr;
\node at (0.9, -0.4){$s(3,2)$};
\end{scope}

\begin{scope}[xshift=8.7cm]
\foreach \y in {0,0.300, ..., 1.200} \draw (0.000, \y) \mysqr; 
\foreach \y in {0,0.300, ..., 0.600} \draw (0.300, \y) \mysqr; 
\foreach \y in {0,0.300, ..., 0.600} \draw (0.600, \y) \mysqr; 
\foreach \y in {0,0.300} \draw (0.900, \y) \mysqr; 
\foreach \y in {0,0.300} \draw (1.200, \y) \mysqr; 
\draw (1.500,0) \mysqr;
\draw (1.800,0) \mysqr;
\node at (1.05, -0.4){$s(4,1)$};
\end{scope}
\end{tikzpicture}
\caption{The first staircase bases of $\mbox{IPM}_2$} \label{fig:bases-ipm}
\end{figure}

These staircases are clearly stable by all rules of $\mbox{IPM}_k(n)$. We define analogously the staircase basis $B_k=\{ s(w,l) \mid w > 0, 1 \leq l \leq k \}$. It is clear that $s(w,l) \leq s(w',l')$ (sequence order) if and only if $(w,l)$ is not larger than $(w',l')$ in lexicographical order. We can see that $\leq$ is a linear order over $B_k$. For $t \in \mbox{IPM}_k(n)$, we define $sw(t)=(w,l)$ such that $s(w,l) \leq t$, and for any $s \in B$ with $s \leq t$, we have $s \leq s(w,l)$. For example, for $t=(8,8,5,5) \in \mbox{IPM}_2(26)$, we have $sw(t)=(2,2)$ and $s(sw(t)) = (2,2,1,1)$, while, if the same $t$ is seen as a configuration of $\mbox{IPM}_5(26)$, we have $sw(t) = (1,4)$ and $s(sw(t)) = (1,1,1,1)$. As before, with respect to the sequence order, $s(sw(t))$ is the largest staircase among all those that are smaller than $t$. It is clear that we cannot apply $\mbox{SLIDE}_k$ on any $s(w,l)$. We also have the following analogue of Proposition \ref{thm:bw-monotone}.

\begin{thm}
For $t,t' \in \mbox{IPM}_k(n)$ and $t\to t'$, we have $s(sw(t)) \leq s(sw(t'))$.
\end{thm}

\begin{proof}
Set $b = s(sw(t))$. By definition of $sw$, it suffices to prove that $b \leq t'$. Suppose that $t'$ is obtained from $t$ by applying the rule $\mbox{SLIDE}_k$ on column $c$. Therefore, $t_c = t'_c + 1$; for some suitable $p<k$, we have $t_{c+p} = t'_{c+p} - 1$ and $t_{c+p} \leq t_c - 2$; and for all $j \not \in \{c, c+p\}$, we have $t_j = t'_j$. By definition, $b \leq t$. The only column that may prevent $b \leq t'$ is column $c$. However, $b_{c+p} \geq b_{c} - 1$ for any $p < k$ (by definition of $b$), thus $t_{c+p} \geq b_{c+p} \geq b_{c} - 1 $. Since we also have $t'_c = t_c - 1 \geq t_{c+p} +1$, we have $t'_c \geq b_{c}$, which concludes the proof.
\end{proof}

We now propose a few definitions similar to those settled for SPM.

\begin{defi}{\textbf{Staircase width and reduced form for \mbox{IPM}}} \label{def:misc-ipm}

We define $\mbox{IPM}_k(n,w,l)$ as the subset of $\mbox{IPM}_k(n)$ of all elements with staircase width $(w,l)$, or formally $\mbox{IPM}_k(n,w,l) = \{ s \in \mbox{IPM}_k(n) | sw(s) = (w,l) \}$. We can see that the family $\{\mbox{IPM}_k(n,w,l)\}_{w, l}$ is a partition of the set $\mbox{IPM}_k(n)$.

For $s \in \mbox{IPM}_k(n,w,l)$, we say that $red_{(w,l)}(s) = (s_i - b(sw(s))_i)_{i \geq 0}$ is its \emph{reduced form}. By definition, every reduced form is a sequence of natural numbers. We denote by $R_k(n,w,l)$ the set of reduced forms of elements in $\mbox{IPM}_k(n,w,l)$.
\end{defi}

Some examples of $\mbox{IPM}$ configurations and their reduced forms are illustrated in Figure~\ref{fig:red-ipm}.

\begin{figure}[!htb]
\centering
\begin{tikzpicture}
\def \mysqr{rectangle +(0.300000, 0.300000)}

\begin{scope}[xshift=0cm]
\foreach \y in {0,0.300, ..., 2.100} \draw (0.000, \y) \mysqr; 
\foreach \y in {0,0.300, ..., 2.100} \draw (0.300, \y) \mysqr; 
\foreach \y in {0,0.300, ..., 1.200} \draw (0.600, \y) \mysqr; 
\foreach \y in {0,0.300, ..., 1.200} \draw (0.900, \y) \mysqr; 
\foreach \y in {0,0.300} \filldraw[fill=black!50] (0.000, \y) \mysqr; 
\foreach \y in {0,0.300} \filldraw[fill=black!50] (0.300, \y) \mysqr; 
\filldraw[fill=black!50] (0.600,0) \mysqr;
\filldraw[fill=black!50] (0.900,0) \mysqr;
\node[text width=6cm, align=center] at (0.6, -0.8) {$s = (7,7,4,4,0,\ldots)$ \\ Basis: $s(2,2)$ \\ $red_{(2,2)}(s)=(5,5,3,3,0,\ldots)$};
\end{scope}

\begin{scope}[xshift=5.5cm]
\foreach \y in {0,0.300, ..., 1.800} \draw (0.000, \y) \mysqr; 
\foreach \y in {0,0.300, ..., 1.200} \draw (0.300, \y) \mysqr; 
\foreach \y in {0,0.300, ..., 1.200} \draw (0.600, \y) \mysqr; 
\foreach \y in {0,0.300,0.6} \draw (0.900, \y) \mysqr; 
\foreach \y in {0,0.300} \draw (1.200, \y) \mysqr; 
\draw (1.500, 0) \mysqr; 
\draw (1.800,0) \mysqr;
\draw (2.100,0) \mysqr;
\foreach \y in {0,0.300, ..., 1.200} \filldraw[fill=black!50] (0.000, \y) \mysqr; 
\foreach \y in {0,0.300, ..., 0.600} \filldraw[fill=black!50] (0.300, \y) \mysqr; 
\foreach \y in {0,0.300, ..., 0.600} \filldraw[fill=black!50] (0.600, \y) \mysqr; 
\foreach \y in {0,0.300} \filldraw[fill=black!50] (0.900, \y) \mysqr; 
\foreach \y in {0,0.300} \filldraw[fill=black!50] (1.200, \y) \mysqr; 
\filldraw[fill=black!50] (1.500,0) \mysqr;
\filldraw[fill=black!50] (1.800,0) \mysqr;
\node[text width=7cm, align=center] at (1.05, -0.8) {$s = (6,4,4,3,2,1,1,1,0,\ldots)$ \\ Basis: $s(4,1)$ \\ $red_{(4,1)}(s)=(2,1,1,1,0,0,0,1,0,\ldots)$};
\end{scope}

\end{tikzpicture}
\caption{Some examples of reduced form in $\mbox{IPM}_2$} \label{fig:red-ipm}
\end{figure}
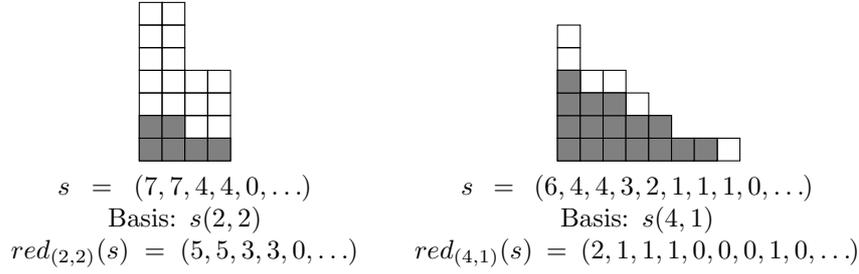

To obtain analogues of the decomposition theorem, we start from the characterization of elements in $\mbox{IPM}_k(n)$. The following characterization of $\mbox{IPM}(n)$ is first given in \cite{goles2002sandpiles}, then in \cite{massazza2010ipm} it is used to give an exhaustive generation algorithm for $\mbox{IPM}(n)$. We adapt the following notations from \cite{massazza2010ipm}. We denote by $p^{[n]}$ the sequence $(p, \ldots, p)$ of $n$ elements equals to $p$ and we recall that the concatenation operation of two sequences is denoted by $\cdot$.

\begin{thm} \label{thm:ipm-chara}
A partition $s$ is in $\mbox{IPM}_k(n)$ for a certain $n$ if and only if it does not contain the following forbidden patterns (for $p>0$ and $h>1$):
\begin{itemize}
\item $p^{[k+2]}$
\item $(p+1)^{[k+1]} \cdot p^{[k+1]}$
\item $(p+h)^{[k+1]} \cdot \prod_{i=1}^{h-1} (p+h-i)^{[k]} \cdot p^{[k+1]}$
\end{itemize}
\end{thm}

Using this characterization, we will prove an analogue of Proposition \ref{prop:bw-width} for the ice pile model.

\begin{lem} \label{lem:bw-seq-order-ipm}
The largest (in the sequence order $\leq$) $\mbox{IPM}_k$ configuration $s$ with $s_0 \leq w$ is $w \cdot s(w,k)$.
\end{lem}

\begin{proof}
Let $s$ be an $\mbox{IPM}_k$ configuration with $s_0 \leq w$. We have $s_i \leq w$ for all $i \geq 0$. Theorem~\ref{thm:spm-chara} implies that, for any integers $i \geq 0$ and $p > 0$, $s_{i+pk+1} \leq s_{i}-p$, thus for any $0 \leq i < k$ and $p > 0$, $s_{pk+i+1} \leq w-p$. We observe that this is equivalent to $s \leq w \cdot s(w)$. We conclude by noticing that $w \cdot s(w,k)$ is also an $\mbox{IPM}_k$ configuration.
\end{proof}

\begin{prop} \label{prop:bw-width-ipm}
For $s \in \mbox{IPM}_k(n,w,l)$, we have $s_{l+k(w-1)+1}=0$.
\end{prop}

\begin{proof}
For $s \in \mbox{IPM}_k(n,w,l)$, we have $sw(s)=(w,l)$ and we denote by $b = s(sw(s))$ the staircase base of configuration $s$. By definition of $b$, there is $i = l + kp \geq 0$ for a certain $p$ such that $s_{i} = b_{i} = w-p-1$. Since suffix $(s_{i}, s_{i+1}, \ldots)$ is also an $\mbox{IPM}_k$ configuration, by applying Lemma~\ref{lem:bw-seq-order-ipm}, we know that it is smaller than $w-p-1 \cdot s(w-p-1,k)$, thus $s_{l+k(w-1)+1} = s_{i+(w-p-1)k+1} = 0$.
\end{proof}

This proposition means that every reduced form $r \in R_k(n,w,l)$ is in fact a $(l+k(w-1)+1)$-tuple of natural numbers. We will now characterize elements in $R_k(n,w,l)$ by the following analogue of Proposition \ref{prop:reduceform-chara}.

\begin{prop} \label{prop:reducedform-chara-ipm}
A $(l+k(w-1)+1)$-tuple $r=(r_0, r_1, \ldots, r_{l+k(w-1)})$ of natural numbers is in $R_k(n,w,l)$ if and only if
\begin{itemize}
\item \textbf{(i)} $\sum_{i=0}^{l+k(w-1)} r_i = n - lw - kw(w-1)/2$;
\item \textbf{(ii)} There exists an index $0 < i_0 \leq l + k(w-1)$ of the form $l + kp_0$ (for some integer $p_0$) such that $r_{i_0}=0$;
\item \textbf{(iii)} For all $i \geq 0$ of the form $l+kp$ (for a certain integer $p$), we have $r_i \geq r_{i+1} \geq \ldots \geq r_{i+k-1} \geq r_{i+k} - 1$;
\item \textbf{(iv)} For all $i \geq 0$ and $j = i + kp + 1$ (for a certain integer $p>0$), we have $r_i \geq r_j - 1$ when $i \equiv l-1 \pmod k$, and $r_i \geq r_j$ otherwise.
\end{itemize}
\end{prop}

\begin{proof}
Let $s \in \mbox{IPM}_k(n,w,l)$ and $r=red_{(w,l)}(s)$. Conditions (i), (ii) and (iii) come from the definitions of $\mbox{IPM}_k(n,w,l)$ and of $red_{(w,l)}$. For the condition (iv), we notice that Theorem~\ref{thm:ipm-chara} implies that for any integers $i \geq 0$ and $j = i + kp + 1$ (for a certain integer $p>0$), $s_{j} \leq s_{i}-p$. Since we have $s(w,l)_i = s(w,l)_j + p + 1$ when $i \equiv l-1 \pmod k$ and $s(w,l)_i = s(w,l)_j + p$ otherwise, we easily verify that condition (iv) holds.

Conversely, let $r$ be a $(l+k(w-1)+1)$-tuple that verifies conditions (i), (ii), (iii) and (iv), and $t=r+s(w,l)$. It follows from conditions (i) and (iii) that $t$ is a partition of $n$, and it suffices to prove that $t \in \mbox{IPM}_k(n)$, because condition (ii) will ensure that $t$ has the correct basis. Suppose that there is a forbidden pattern in $t$ between column $c$ and $c+kp+1$ for some $p>0$, and we have $t_{c+kp+1} = t_c - p + 1$. If $c \equiv l-1 \pmod k$, we have $s(w,l)_c = s(w,l)_{c+kp+1} + p + 1$, thus $r_c = r_{c+kp+1} - 2$; otherwise, we have similarly $r_c = r_{c+kp+1} - 1$. This cannot happen when (iv) is verified, thus $t \in \mbox{IPM}_k(n)$.
\end{proof}

We notice that conditions (i) and (ii) ensure that the reduced form is in the $R_k(n,w,l)$ with correct parameters. Conditions (iii) and (iv) are stable by prefix-taking, the same as in the case of \mbox{SPM}. However, these two conditions, and also the length of tuple, are parametrized by $w$ and $l$. Simply taking prefix will preserve (iii) and (iv), but with parameters $w,l$ not compatible with the length of tuple. Therefore, if we mimic the decomposition theorem for \mbox{SPM} in a naive way, the part before the first zero will not be a valid reduced form. We try to circumvent this problem by extending our definition of reduced form.

\begin{defi}{\textbf{Extended reduced forms, augmented reduced forms}} \label{def:extended-reduced-form-ipm}
For a pair of positive integers $(w,l)$ and a tuple $t$ of length $l+k(w-1)+1$, we say that $t$ is an \emph{extended reduced form} if $t$ verifies conditions (ii), (iii) and (iv) in Proposition \ref{prop:reducedform-chara-ipm}, and $t$ is called an \emph{augmented reduced form} if conditions (iii) and (iv) are verified.

We denote by $R'_k(w,l)$ the set of extended reduced forms, and $A_k(w,l)$ the set of augmented reduced forms. Clearly we have $R'_k(w,l) \subset A_k(w,l)$.
\end{defi}

Clearly the subset of all $(l+k(w-1)+1)$-tuples in $R'_k(w,l)$ is exactly the union of $R_k(n,w,l)$ for all possible $n$. 

We now introduce a function that will be used to turn augmented reduced forms into the more regular extended reduced forms.

\begin{defi} \label{def:socle-moving-function}
For positive integers $k$ and $l$ such that $0 \leq l < k$, we define the function $pl_{k}^{l}$ on tuples with arbitrary length of non-negative integers as follows: set $r'=pl_{k}^{l}(r)$, we define $r'_i=r_i-1$ for $i \equiv l \pmod k$ and $r'_i=r_i$ otherwise. This function is undefined when there exists some $i \equiv l \pmod k$ such that $r_i=0$ .
\end{defi}

\begin{figure}
\centering
\begin{tikzpicture}
\def \mysqr{rectangle +(0.3,0.3)}
\draw[->] (3,0.9) -- (4.7,0.9) node[above,midway] {$pl_3^2$};

\foreach \x in {0,0.3,...,2.4} \filldraw[black!50] (\x,0) \mysqr;
\foreach \x in {0,0.3,...,1.5} \filldraw[black!50] (\x,0.3) \mysqr;
\foreach \x in {0,0.3} \filldraw[black!50] (\x,0.6) \mysqr;

\foreach \x in {5,5.3,...,7.7} \filldraw[black!50] (\x,0) \mysqr;
\foreach \x in {5,5.3,...,6.8} \filldraw[black!50] (\x,0.3) \mysqr;
\foreach \x in {5,5.3,5.6} \filldraw[black!50] (\x,0.6) \mysqr;

\foreach \x in {0,0.3,...,3} \draw (\x,0) \mysqr;
\foreach \x in {0,0.3,...,2.1} \draw (\x,0.3) \mysqr;
\foreach \x in {0,0.3,...,1.2} \draw (\x,0.6) \mysqr;
\foreach \x in {0,0.3,...,1.2} \draw (\x,0.9) \mysqr;
\foreach \x in {0,0.3,...,0.9} \draw (\x,1.2) \mysqr;
\foreach \x in {0,0.3} \draw (\x,1.5) \mysqr;
\draw (0,1.8) \mysqr;
\node[text centered] at (1,-0.55) {$r = (4,3,3,2,0,1,1,0,1,1) \in A_3(3,2)$};

\foreach \x in {5,5.3,...,8} \draw (\x,0) \mysqr;
\foreach \x in {5,5.3,...,7.1} \draw (\x,0.3) \mysqr;
\foreach \x in {5,5.3,...,6.2} \draw (\x,0.6) \mysqr;
\foreach \x in {5,5.3,...,6.2} \draw (\x,0.9) \mysqr;
\foreach \x in {5,5.3,...,5.9} \draw (\x,1.2) \mysqr;
\foreach \x in {5,5.3} \draw (\x,1.5) \mysqr;
\draw (5,1.8) \mysqr;
\node[text centered] at (7,-0.55) {$r' = (4,3,2,2,0,0,1,0,0,1) \in A_3(3,3)$};
\end{tikzpicture}
\caption{An example of application of $pl_k^l$, with $k=3, l=2$}
\label{figure:peel-ipm}
\end{figure}

The example in Figure \ref{figure:peel-ipm} shows graphically the effect of $pl_k^l$. Intuitively, if $r$ is the augmented reduced form of some $\mbox{IPM}_k$ configuration with respect to the basis $s(w,l)$, then $pl_k^l(r)$ is the augmented reduced form of the same $\mbox{IPM}_k$ configuration with respect to the next basis, in the linear order for $B_k$.

We now investigate some properties of $pl_{k}^{l}$ in the following lemma.

\begin{lem} \label{lemma:socle-moving-properties}
The function $pl_{k}^{l}$ verifies the following properties.
\begin{enumerate}
\item The function $pl_{k}^{l}$ is undefined on the set of extended reduced form $R'_k(w,l)$.
\item For a tuple $r \in A_k(w,l)$, if $pl_{k}^{l}(r)$ is defined, we have $pl_{k}^{l}(r) \in A_k(w',l')$, where $w'=w+1$, $l'=1$ if $l=k$, and $w'=w$, $l'=l+1$ otherwise.
\item For a tuple $r \in A_k(w,l)$, we recursively define the sequence of tuples $r = r^{(0)}, r^{(1)}, \ldots$  by $r^{(i+1)} = pl_{k}^{l^{(i)}}(r^{(i)}) \in A_k(w^{(i+1)},l^{(i+1)})$. This sequence becomes undefined after a certain index $c$ satisfying $r^{(c)} \in R'_k(w^{(c)},l^{(c)})$. Moreover, $(w^{(i)},l^{(i)})$ does not depend on $r$.
\item In the case of the previous assertion, we say that $r=Aug_{w,l}(r^{(c)},c)$ is equal to $r^{(c)}$ augmented by $c$. Regarded as a function, $Aug_{w,l}$ is a bijection between $\{ (r',c) | r' \in R'_k(w^{(c)}, l^{(c)}), c \in \mathbb{N} \}$ and $A_k^{w,l}$.
\end{enumerate}
\end{lem}

\begin{proof}
The first assertion follows from the definitions of both $R'_k(w,l)$ and $pl_{k}^{l}$. More precisely, the elements of $R'_k(w,l)$ verify condition (ii) in Proposition \ref{prop:reducedform-chara-ipm}, which prevents $pl_{k}^{l}$ to be defined.

The second assertion comes from simple verification of conditions (iii) and (iv) in Proposition \ref{prop:reducedform-chara-ipm} for $pl_{k}^{l}(r)$.

We now prove the third assertion using the result of the second. We will first prove that the sequence terminates, then discuss the properties of $r^{(c)}$ and $(w^(i), l^{(i)})$.

To show that the process terminates, we notice that $pl_{k}^{l}$ is always a decreasing function with respect to the sequence order, and strictly decreasing in the case $l=1$, since a tuple must have its first element. If the sequence $r^{(0)} = r, r^{(1)}, \ldots$ does not terminate, the case $l=1$ will occur an infinite number of times, thus we can extract an infinite strictly decreasing sequence from the original one. This contradicts the well-foundedness of the sequence order of tuples of non-negative integers with fixed length. Thus termination of the process follows.

For the iterative process to terminate, $pl_{k}^{l^{(c)}}$ must be undefined on $r^{(c)}$, that is to say condition (ii) is verified, following the same reasoning as in the first assertion. Combining with $r^{(c)} \in A_k(w^{(c)},l^{(c)})$, we have $r^{(c)} \in R'_k(w^{(c)},l^{(c)})$. The independence of $(w^{(i)},l^{(i)})$ from $r$ is implied by the independence of $w', l'$ from $r$ in the second assertion.

For the last assertion, $Aug_{w,l}$ is clearly surjective. We also notice that, given a pair $(r',c)$ with $r' \in R'_k(w^{(c)}, l^{(c)})$, it is easy to uniquely reconstruct a tuple $r$ such that $r=Aug_{w,l}(r^{(c)},c)$ by reversing the recursive process indexed by the sequence $(w^{(0)},l^{(0)}), \ldots, (w^{(c)},l^{(c)})$ independent of $r'$. Therefore, $A_k^{w,l}$ is also injective, which proves the assertion.
\end{proof}

We notice that $s(w,l) < s(w',l')$ (as defined in the second assertion of Lemma \ref{lemma:socle-moving-properties}) are consecutive elements in the linear order $B_k$ with respect to the sequence order. In fact, let $m$ be the length of tuple $r$, and $s(w,l)|_{m}$ the prefix of $s(w,l)$ of length $m$, then we can easily verify that $r + s(w,l)|_{m} = r' + s(w',l')|_{m}$, where addition is intended as pointwise. This equality means that the function $pl_{k}^{l}$ transforms an augmented reduced form on a certain socle to another augmented reduced form on the smallest socle that covers the previous one, and these two augmented reduced forms are equivalent in the sense that they actually give the same prefix of a configuration, but with the removal of different staircases.

With all these modifications, we can state an analogue of our \mbox{SPM} reduced form decomposition theorem.

\begin{thm} \label{thm:rec-decomp-red-ipm}
An extended reduced form $r \in R'_k(w,l)$ can be uniquely decomposed into the following form:
\[ r = (t_0, \ldots, t_{l+kp-1}, 0, u_0, \ldots, u_m), \]
with some integer $p$ such that $t = (t_0, \ldots, t_{l+kp-1})$ and $u=(u_0, \ldots, u_m)$ verify the following conditions:
\begin{enumerate}
\item If $u$ is not empty, we have $u_i \in \{ 0,1 \}$ for $i \equiv -1 \pmod k$, and $u_i=0$ otherwise.
\item We have $t \in A_k(w,l)$, but $t \notin R'_k(w,l)$.
\end{enumerate}
This decomposition will be denoted as $r=((t',c),p,u)$, with $t=Aug_{w,l}(t^{(c)},c)$ and $t'=t^{(c)}$. When $t$ is empty, we take $c=0$.
\end{thm}

\begin{proof}
The existence of index $i=l+kp$ such that $r_i=0$ is given by condition (ii). To ensure $t \notin R'_k(w,l)$, we take the smallest such index. By condition (iii), we know that $u_i \in \{ 0,1 \}$ for $i \equiv -1 \pmod k$ and $u_i=0$ otherwise, and this does not violate condition (iv). It is clear that $t \in A_k(w,l)$, since condition (iii) and (iv) are invariant under prefix-taking. Therefore, from the third assertion of Lemma \ref{lemma:socle-moving-properties}, we have the existence of $(t^{(c)},c)$ as a representation of $t$. We now have the validity of this decomposition. The uniqueness is provided by the uniqueness of smallest index $i=l+kp$ such that $r_i=0$, and the uniqueness of the pair $(r^{(c)},c)$ giving $r=Aug_{w,l}(r^{(c)},c)$ provided in the last assertion of Lemma \ref{lemma:socle-moving-properties}.
\end{proof}

Since $R_k(n,w,l)$ is the subset of $R'_k(w,l)$ consisting of all $(l+k(w-1)+1)$-tuples with correct weight, we can use this decomposition theorem of extended reduced forms to enumerate and to generate \mbox{IPM} configurations, following the same approach as previously done for \mbox{SPM}.

\section{Future work}
We would like to extend this approach to more general sand pile models, for example \mbox{BSPM}, a bi-dimensional version of \mbox{SPM}. However, this extension does not seem to be easy. There are several difficulties. Firstly, rules now involve two directions, which weakens the foundation of staircase bases on well-behaving rules. Secondly, we do not yet have a good characterization of configurations in \mbox{BSPM}, even for stable ones. Lastly, simulations show that fixed points in \mbox{BSPM} have a great variety of different shapes, which would be difficult to approximate using a ``small'' set of staircase bases.

\section*{Acknowledgement}
We deeply appreciate the anonymous reviewers and the referee for their precious corrections and improvements, which made our work simpler and more accessible.

\bibliographystyle{plain}
\bibliography{spm-jn-rev}
\end{document}